\documentclass[12pt]{article}
\usepackage{amsmath, mathrsfs}
\usepackage{amsfonts}
\usepackage{amsthm}
\usepackage{graphicx}
\graphicspath{{Figures/}} 
\usepackage{overpic}
\usepackage{amssymb} 
\usepackage{pictex}
\usepackage{rotating}  
\usepackage{color}
\usepackage{cite} 

\usepackage{enumitem} 

\usepackage{accents} 

\usepackage{yhmath} 
\usepackage{etex} 

\usepackage{comment}
\numberwithin{equation}{section}

\newtheorem{theorem}{Theorem}[section]
\newtheorem{proposition}[theorem]{Proposition}
\newtheorem{lemma}[theorem]{Lemma}
\newtheorem{corollary}[theorem]{Corollary}
\newtheorem{Definition}[theorem]{Definition}
\newenvironment{definition}{\begin{Definition}\rm}{\end{Definition}}
\newtheorem{Remark}[theorem]{Remark}
\newenvironment{remark}{\begin{Remark}\rm}{\end{Remark}}
\newtheorem{RHproblem}[theorem]{RH problem}

\newtheorem{Example}[theorem]{Example}

\newcommand{\C}{\mathbb{C}}

\newcommand{\N}{\mathbb{N}}

\newcommand{\R}{\mathbb{R}}

\usepackage{color}

\def\det{\mathop{\mathrm{det}}\nolimits}

\renewcommand{\bar}{\overline}
\renewcommand{\tilde}{\widetilde}
\renewcommand{\hat}{\widehat}

\usepackage[bookmarksopen, naturalnames]{hyperref}
\begin{document}
\title{Random Polynomials in Several Complex Variables}
\author{T. Bayraktar, T. Bloom and N. Levenberg}
\date{}

\maketitle

\begin{abstract}

We generalize some previous results on random polynomials in several complex 
variables. A standard setting is to consider random polynomials $H_n(z):=\sum_{j=1}^{m_n} a_jp_j(z)$ that are linear combinations of basis polynomials $\{p_j\}$ with i.i.d. complex random variable coefficients $\{a_j\}$ where $\{p_j\}$ form an orthonormal basis for a Bernstein-Markov measure on a compact set $K\subset \C^d$. Here $m_n$ is the dimension of $\mathcal P_n$, the holomorphic polynomials of degree at most $n$ in $\C^d$. We consider more general bases $\{p_j\}$, which include, e.g., higher-dimensional generalizations of Fekete polynomials. Moreover we allow $H_n(z):=\sum_{j=1}^{m_n} a_{nj}p_{nj}(z)$; i.e., we have an {\it array} of basis polynomials $\{p_{nj}\}$ and random coefficients $\{a_{nj}\}$. This always occurs in a weighted situation. We prove results on convergence in probability and on almost sure convergence of $\frac{1}{n}\log |H_n|$ in $L^1_{loc}(\C^d)$ to the (weighted) extremal plurisubharmonic function for $K$. We aim for weakest possible sufficient conditions on the random coefficients to guarantee convergence. 
\end{abstract}

\bigskip


\bigskip

\section{Introduction} We will consider random polynomials in $\C^d$ of degree at most $n$ of the form 
\begin{equation}\label{probhb}H_n(z):=\sum_{j=1}^{m_n} a_{nj}p_{nj}(z)\end{equation}
where $\{a_{nj}\}_{j=1,...,m_n; \ n=1,2,...}$ are an array of i.i.d. random variables; $\{p_{nj}\}_{j=1,...,m_n}$ form a basis for $\mathcal P_n$, the holomorphic polynomials of degree at most $n$ in $\C^d$; and $m_n=$dim$(\mathcal P_n$). The random variables $a_{nj}$ are defined on a probability space $(\Omega, \mathcal H, \mathbb P)$. This occurs naturally in a weighted setting; c.f., subsection 2.2. A special case is when 
\begin{equation}\label{probh}H_n(z):=\sum_{j=1}^{m_n} a_jp_j(z)\end{equation}
where the $\{a_j\}$ are a sequence of i.i.d. complex random variables with a basis $\{p_1,...,p_{m_n}\}$ for $\mathcal P_n$. We always assume the random variables are nondegenerate; i.e., their distribution is not supported at a single point. In the situation where one starts with $K\subset \C^d$ compact and $\mu$ a probability measure on $K$ such that $||p||_{L^2(\mu)}=0$ for a polynomial $p$ implies $p\equiv 0$, a well-studied situation is to let $\{p_j\}_{j=1,...,m_n}$ in (\ref{probh}) be an orthonormal basis in $L^2(\mu)$ for $\mathcal P_n$ constructed via Gram-Schmidt applied to a monomial basis of $\mathcal P_n$. In this setting, convergence of the sequence of (deterministic) plurisubharmonic (psh) functions $\{\frac{1}{2n}\log B_n(z)\}$, where $B_n(z):=\sum_{j=1}^{m_n}|p_j(z)|^2$, to the extremal psh function $V_K$ associated to $K$ plays a major role.

We aim to generalize some of the results occurring in \cite{ba}, \cite{Bay} and \cite{blrp} on asymptotics of normalized zero measures and currents associated to $H_n$ in two ways.  First, we allow our polynomial bases to be more general; they need not come from orthonormal polynomials associated to a single measure $\mu$ on $K$. For example, we allow a basis $\{p_{nj}\}_{j=1,...,m_n}$ for $\mathcal P_n$ for $n=1,2,...$ to come from a sequence of {\it asymptotically (weighted) Bernstein-Markov measures} $\{\mu_n\}$ for $K$ (with weight $Q$). In the sequence case, we allow $\{p_j\}$ which are {\it asymptotically Chebyshev} for $K$. Secondly, we aim for the weakest conditions on the random variables $a_j$ or $a_{nj}$ in order to have the appropriate convergence as, e.g., in \cite{DD} and \cite{BD}. 

For the classical Kac ensemble $H_n(z)=\sum_{j=0}^n a_jz^j=a_n\prod_{k=1}^n(z-z_{nk})$ in one variable, Ibragimov and Zaporozhets \cite{IZ} showed that the condition $\Bbb{E}(\log(1+|a_j))<\infty$ is a necessary and sufficient condition for the sequence $\mu_{H_n}:=\frac{1}{n}\sum_{k=1}^n \delta_{z_{kn}}$ to converge weakly almost surely to normalized arclength measure $\frac{1}{2\pi}d\theta$ on the unit circle $T$. This is the potential-theoretic equilibrium measure $\mu_T$ of $T$; and the monomials $\{z^j\}$ are an orthonormal basis for $L^2$ of the same measure. Kabluchko and Zaporozhets \cite{KZ} proved some results for convergence in probability of $\mu_{H_n}\to \mu_T$ weakly. Pritsker and Ramachandran \cite{PR} gave some almost sure convergence results for special polynomial bases on certain compacta in $\C$. Bloom and Dauvergne \cite{BD} showed that for a regular compact set $K\subset \C$ with equilibrium measure $\mu_K$, given a Bernstein-Markov measure $\mu$ on $K$ and an orthonormal basis of polynomials $\{p_j\}$ for $L^2(\mu)$, if $\mathbb P(|a_j|>e^{|z|})=o(1/|z|)$, then $\mu_{H_n}\to \mu_K$ weakly in probability. Dauvergne generalized these univariate results in \cite{DD}: for his class of asymptotically minimal polynomials $\{p_j\}$ for $K$, which includes orthonormal polynomials for a Bernstein-Markov measure $\mu$ on $K$ and essentially all ``classical'' sets of polynomials (Chebyshev, Fekete, Fejer, etc.), the condition $\Bbb{E}(\log(1+|a_j))<\infty$ is a necessary and sufficient condition for the sequence $\mu_{H_n}\to \mu_K$ weakly almost surely while the condition $\mathbb P(|a_j|>e^{|z|})=o(1/|z|)$ is a necessary and sufficient condition for the sequence $\mu_{H_n}\to \mu_K$ weakly in probability.

We work in $\C^d$ with $d>1$. Bloom and Shiffman \cite{BS}, building upon earlier work of Shiffman and Zelditch \cite{SZ}, proved some results for Gaussian random polynomials where the basis polynomials are orthonormal polynomials for a Bernstein-Markov measure $\mu$ on $K\subset \C^d$. In this orthonormal polynomial setting, as well as a weighted one, the authors (\cite{ba}, \cite{Bay} and \cite{blrp}) have given sufficient conditions on $\{a_j\}$ or $\{a_{nj}\}$ in order to get certain almost sure convergence results for normalized zero currents associated to random $\{H_n\}$ in (\ref{probhb}) or in (\ref{probh}). In the next subsection, we define a notion of asymptotically Chebyshev polynomials for $K$ from \cite{Bloom} which mimics the $L^{\infty}-$version of Dauvergne's asymptotically minimal polynomials in \cite{DD}. For certain bases, we prove a determinstic result on the corresponding sequence of deterministic psh functions $\{\frac{1}{2n}\log B_n\}$. In the second part of section 2 we prove local uniform convergence of $\frac{1}{2n}\log B_n$ to $V_K$ (or $V_{K,Q}$) in the setting of asymptotically (weighted) Bernstein-Markov measures. In sections 3 and 4 we prove convergence in probability and almost sure convergence results for our random $\{H_n\}$ using bases from these settings under essentially weakest possible conditions on the arrays of i.i.d. random coefficients $\{a_{nj}\}$. Here we utilize some probabilistic arguments from \cite{BD}.
\vskip4pt

\noindent{\bf Acknowledgements}: We wish to thank Vu Duc Viet for pointing out an error in a previous version of this paper and Bela Nagy for the example in Remark \ref{2.3}.

\section{Deterministic results} 

\subsection{Asymptotically Chebyshev polynomials}

As in the introduction we let $m_n= \ $dim$(\mathcal P_n$). Let 
$$L(\C^d):=\{u \in PSH(\C^d): u(z)\leq \log |z|+0(1), \ |z| \to \infty\},$$
the Lelong class of psh functions of minimal growth. Given a nonpluripolar compact subset $K\subset \C^d$,  let 
\begin{equation} \label{vk} V_{K}(z):= \sup \{u(z): u\in L(\C^d), \ u\leq 0 \ \hbox{on} \ K\} \end{equation}
and $V_K^*(z):=\limsup_{\zeta\to z}V_K(\zeta)$. Then $V_K^*\in L(\C^d)$ and it is known that
$$V_{K}(z)=\sup \{\frac{1}{deg(p)}\log |p(z)|: p\in \cup_n \mathcal P_n, \ ||p||_K\leq 1\}.$$
In particular, $V_K=V_{\hat K}$ where $\hat K:=\{z\in \C^d: |p(z)|\leq ||p||_K, \ p\in \cup_n \mathcal P_n\}$. Moreover, given a {\it Bernstein-Markov} (BM) measure $\mu$ on $K$, i.e., $||p_n||_K\leq M_n ||p_n||_{L^2(\mu)}$ for $p_n \in \mathcal P_n$ with $M_n^{1/n}\to 1$, it is straightforward to see that for the sequence of Bergman functions $B_n(z):=\sum_{j=1}^{m_n}|p_j(z)|^2$, where $\{p_j\}_{j=1,...,m_n}$ is an orthonormal basis of $\mathcal P_n$ in $L^2(\mu)$, we have $\frac{1}{2n}\log B_n(z)\to V_K$ pointwise on $\C^d$. Moreover the convergence is locally uniform if $K$ is regular; i.e., if $V_K$ is continuous (cf. Proposition 3.1, \cite{blrp}). We generalize this in Propositions \ref{locunif} and \ref{locunifb} in subsection 2.2. For more on (BM) measures, see \cite{BMsur}.

For simplicity, throughout this subsection we assume our compact sets are regular. We fix an ordering $\prec$ on $\N^d$ so that $\alpha \prec \beta$ if $|\alpha|\leq |\beta|$. Writing 
$$e_i(z)=z^{\alpha}=z^{\alpha(i)}:=z_1^{\alpha_1}\cdots z_d^{\alpha_d}$$ and $|\alpha|:=\alpha_1+\cdots +\alpha_d$, let 
$$P(\alpha)=P(\alpha(i)):=\{e_i(z)+\sum_{j<i}c_je_j(z): c_j \in \C \}=\{z^{\alpha(i)}+\sum_{\beta \prec \alpha(i)} c_{\beta}z^{\beta}\}.$$
Given a compact set $K\subset \C^d$ we define the Chebyshev constants
$$\tau_{\alpha(i)}(K):= \inf \{||p||_K:p\in P(\alpha(i))\}.$$
Note this generalizes the univariate notion; if $d=1$ and $K\subset \C$, the Chebyshev polynomial of degree $j$ for $K$ is the monic polynomial $t_j(z)=z^j+...$ of degree $j$ of minimal supremum norm on $K$ among all such polynomials. 
Following \cite{Z}, we use these to define {\it directional Chebyshev constants}. We have $\frac{\alpha(i)}{|\alpha(i)|}\in \Sigma$, the standard
$(d-1)-$simplex  in $\R^d$:
$$\Sigma = \{\theta =(\theta_1,...,\theta_d)\in \R^d: \sum_{j=1}^d\theta_j=1, \ \theta_j\geq 0, \ j=1,...,d\}.$$
Let $\Sigma^o :=\{\theta \in \Sigma:  \ \theta_j > 0, \ j=1,...,d\}$. 
Zaharjuta \cite{Z} showed that for $\theta \in \Sigma^o$, the limit
\begin{equation}\label{zah} \tau(K,\theta):=\lim_{|\alpha(i)|\to \infty, \ \frac{\alpha(i)}{|\alpha(i)|}\to \theta} \tau_{\alpha(i)}(K)^{1/|\alpha(i)|} \ \hbox{exists};\end{equation}
and we call $\tau(K,\theta)$ the directional Chebyshev constant of $K$ in the direction $\theta$.

Following \cite{Bloom}, given $\theta\in \Sigma^o$, we call a sequence of polynomials $\{q_k\}_{k\in Y_{\theta}}$, $Y_{\theta}\subset \N$, where $q_k(z)=z^{\beta(k)}+\sum_{\beta \prec \beta(k)}c_{\beta}z^{\beta}$ {\it $\theta-$asymptotically Chebyshev for $K$} if 
$$\lim_{k\to \infty, \ k\in Y_{\theta}} deg(q_k)=\infty, \ \lim_{k\to \infty, \ k\in Y_{\theta}}\beta(k)/|\beta(k)|= \theta,$$ and 
$$\lim_{k\to \infty, \ k\in Y_{\theta}} ||q_k||_K^{1/|\beta(k)|}=\tau(K,\theta).$$
If we need to specify the multiindex and degree we will write $q_j$ as $q_{n,\alpha}$  where $|\alpha|=n=deg(q_{n,\alpha})$. Here the ordering $\prec$ then gives us the ordering of the sequence $\{q_j\}_j$. 

\begin{definition} \label{good} We call the sequence $\{q_j\}_j$  {\it asymptotically Chebyshev for $K$} if for each  $\theta\in \Sigma^o$, we can find a subsequence $Y_{\theta}\subset \N$ so that $\{q_k\}_{k\in Y_{\theta}}$ is $\theta-$asymptotically Chebyshev for $K$. If the sequence has the additional property that for each $\theta\in \Sigma^o$, {\it every} sequence
of $\alpha \in \N^d$ with $\lim_{|\alpha|\to \infty} \frac{\alpha}{|\alpha|}=\theta$ satisfies $\lim ||q_{n,\alpha}||_K^{1/|\alpha|}=
\tau (K,\theta)$, then we call $\{q_j\}_j$  {\it $Z-$asymptotically Chebyshev for $K$}.
\end{definition}

The $Z$ is in honor of Zaharjuta. We will always assume that there is one $q_j=q_{n,\alpha}$ for each multiindex $\alpha\in \N^d$ with $q_{n,\alpha}(z)=c_{\alpha} z^{\alpha} + ...$ where $c_{\alpha}\not = 0$ so that we have a polynomial basis. Note that if $d=1$, any asymptotically Chebyshev sequence for $K$ is $Z-$asymptotically Chebyshev.

Bloom (Theorem 4.2 of \cite{Bloom}) showed that for $\{q_j\}_j$ asymptotically Chebyshev for $K$
\begin{equation}\label{clike}[\limsup_{j\to \infty} \frac{1}{deg(q_j)} \log \frac{|q_j(z)|}{||q_j||_K}]^* = V_K(z) \ \hbox{for} \ z\not \in \hat K. \end{equation}
Note that since $\alpha \prec \beta$ if $|\alpha|\leq |\beta|$ we have $deg(q_j)\leq deg(q_{j+1})$ in (\ref{clike}).

\begin{remark} \label{bmcomp} It is clear that if $\{q_j\}=\{q_{n,\alpha}\}$ is ($Z-$) asymptotically Chebyshev for $K$ and $\{c_{n,\alpha}\}$ are positive constants such that for all $\epsilon>0$ there exists $n_0=n_0(\epsilon)$ so that for $n\geq n_0$
$$\exp(-\epsilon n)\leq c_{n,\alpha}\leq \exp(\epsilon n),$$
then $\{\tilde q_j\}=\{\tilde q_{n,\alpha}\}$ is ($Z-$) asymptotically Chebyshev for $K$ where $\tilde q_{n,\alpha}:=c_{n,\alpha} q_{n,\alpha}$. In particular, given a (BM) measure $\mu$ for $K$ and any ($Z-$) asymptotically Chebyshev sequence $\{q_j\}$ for $K$ where $q_j\in P(\alpha(j))$, the sequence $\{\tilde q_j\}$, where $\tilde q_j :=q_j/ ||q_j||_{L^2(\mu)}$, is ($Z-$) asymptotically Chebyshev for $K$. 

\end{remark}

We give some examples of $Z-$asymptotically Chebyshev sequences $\{q_j\}=\{q_{n,\alpha}\}$ for a regular compact set $K\subset \C^d$.

\begin{enumerate} 

\item {\it Chebyshev polynomials}: $\{t_j\}=\{t_{n,\alpha}\}$ where $t_j\in P(\alpha(j))$ with $||t_j||_K=\tau_{\alpha(j)}(K):= \inf \{||p||_K:p\in P(\alpha(j))\}$.

\item {\it  $L^2(\mu)-$minimal polynomials for a (BM) $\mu$}: $\{q_j\}=\{q_{n,\alpha}\}$ where, for $\mu$ a (BM) measure for $K$, $q_j\in P(\alpha(j))$ is the $L^2(\mu)-$minimal (orthogonal) polynomial in $P(\alpha(j))$ for $K$; i.e., $||q_j||_{L^2(\mu)}=\inf \{||p||_{L^2(\mu)}:p\in P(\alpha(j))\}$. That this sequence is $Z-$asymptotically Chebyshev for $K$ follows from Zaharjuta's result (\ref{zah}), the (BM) property, and Remark \ref{bmcomp}.

\end{enumerate}

Another class of examples can be given following \cite{Bloom}. Given a triangular array of points $\{\zeta_{sj}\}_{j=1,...,s; \ s=1,2,...}\subset K$ such that $VDM(\zeta_{s1},...,\zeta_{ss})\not = 0$ for all $s$ where
$$VDM(\zeta_{s1},...,\zeta_{ss}):=\det [e_i(\zeta_j)]_{i,j=1,...,s}  $$
$$:= \det
\left[
\begin{array}{ccccc}
 e_1(\zeta_{s1}) &e_1(\zeta_{s2}) &\ldots  &e_1(\zeta_{ss})\\
  \vdots  & \vdots & \ddots  & \vdots \\
e_{s}(\zeta_{s1}) &e_{s}(\zeta_{s2}) &\ldots  &e_{s}(\zeta_{ss})
\end{array}
\right],$$
we define, for $\alpha=\alpha(s)$, 
$$q_{\alpha}(z):=\frac{VDM(\zeta_{s1},...,\zeta_{ss},z)}{VDM(\zeta_{s1},...,\zeta_{ss})}\in P(\alpha(s+1)).$$
We can write
$$q_{\alpha}(z)=z^{\alpha(s+1)}- L_{\alpha(s)}(z^{\alpha(s+1)})$$
where for a function $f$ on $K$, $L_{\alpha(s)}(f)$ is the Lagrange interpolating polynomial for $f$ in the linear span of $\{z^{\alpha(j)}: j=1,...,s\}$ at the points $\zeta_{s1},...,\zeta_{ss}$; i.e., 
$$L_{\alpha(s)}(f)(z):=\sum_{j=1}^s f(\zeta_{sj})l_{sj}(z)$$
where
$$l_{sj}(z):=\frac{VDM(\zeta_{s1},...,\zeta_{s,j-1},z,\zeta_{s,j+1},...,\zeta_{ss},z)}{VDM(\zeta_{s1},...,\zeta_{ss})}$$
are the fundamental Lagrange interpolating polynomials for $\zeta_{s1},...,\zeta_{ss}$ (note $l_{sj}(\zeta_{sk}) =\delta_{jk}$). In particular, for any $p\in P(\alpha(s))$ we have $p=L_{\alpha(s)}(p)$. The Lebesgue constants associated to $\{\zeta_{sj}\}_{j=1,...,s; \ s=1,2,...}$ are the the sequence of numbers $\{\Lambda_{\alpha}\}_{\alpha \in \N^d}$ where $\Lambda_{\alpha}$ is the norm of the projection operator $f\to L_{\alpha}(f)$ for $f\in C(K)$ using the supremum norm $|| \cdot ||_K$; i.e., 
$$\Lambda_{\alpha}:=\sup \{||L_{\alpha}(f)||_K: f\in C(K), \ ||f||_K\leq 1\}.$$
Then 
$$\Lambda_{\alpha(s)}=\sup_{z\in K} \sum_{j=1}^s |l_{sj}(z)|.$$

If we take, for each $s=1,2,...$ points $\zeta_{s1},...,\zeta_{ss}\in K$ to maximize $|VDM(t_1,...,t_s)|$ over all choices of $t_1,...,t_s\in K$, then $|\Lambda_{\alpha(s)}|\leq s$. We call such arrays {\it Fekete arrays} for $K$. Here are further examples of $Z-$asymptotically Chebyshev polynomials for $K$.

\begin{enumerate}

\item[3.] {\it Fekete polynomials}: $\{q_j\}$ where $q_j(z)=z^{\alpha(j)}- L_{\alpha(j-1)}(z^{\alpha(j)})$ are polynomials associated to an array of Fekete points for $K$. This is a special case of our next example: 

\item[4.] {\it Polynomials associated to arrays in $K$ with Lebesgue constants of subexponential growth}: $\{q_j\}$ where $q_j(z)=z^{\alpha(j)}- L_{\alpha(j-1)}(z^{\alpha(j)})$ are polynomials associated to an array $\{\zeta_{sj}\}_{j=1,...,s; \ s=1,2,...}\subset K$ with $\lim_{|\alpha|\to \infty} \Lambda_{\alpha}^{1/|\alpha|}=1$.

\end{enumerate}

\noindent That these are examples of $Z-$asymptotically Chebyshev polynomials for $K$ follows from Corollary 4.4 of \cite{Bloom}: for such $\{q_j\}$, since $q_j, t_j \in P(\alpha(j))$, we have $$q_j(z)=z^{\alpha(j)}- L_{\alpha(j-1)}(z^{\alpha(j)})= t_j(z)- L_{\alpha(j-1)}(t_j(z))$$ so that
$$||t_j||_K\leq ||q_j||_K=||t_j(z)- L_{\alpha(j-1)}(t_j(z))||_K \leq (1+\Lambda_{\alpha(j)})||t_j||_K.$$

Given an asymptotically Chebyshev sequence $\{q_j\}$ for $K$, let $p_j(z):= |q_j(z)|/||q_j||_K$. We will also use the notation $p_{n,\alpha}$ for $p_j$ if need be, and we set $q_1(z)=p_1(z)\equiv 1$. In analogy with the case of an orthonormal basis associated to a Bernstein-Markov measure, we define, for $n=1,2,...$, the function 
\begin{equation} \label{bn} B_n(z):=\sum_{j=1}^{m_n} |p_j(z)|^2. \end{equation}
It is easy to see that 
\begin{equation} \label{upbound} \limsup_{n\to \infty} \frac{1}{2n}\log B_n(z)\leq V_K(z) \ \hbox{for} \ z\in \C^d  \end{equation}
and
\begin{equation} \label{onK} \lim_{n\to \infty} \frac{1}{2n}\log B_n(z)=0 \ \hbox{for} \ z\in \hat K. \end{equation}
Indeed, since $||p_j||_K=1$, by definition of $V_K$, $|p_j(z)|\leq e^{deg(p_j)V_K(z)}$. Thus
$$B_n(z)\leq \sum_{j=1}^{m_n} e^{2deg(p_j)V_K(z)}\leq m_ne^{2nV_K(z)}$$
which gives (\ref{upbound}) since $m_n=0(n^d)$. Note that this estimate also shows that the family $\{\frac{1}{2n}\log B_n\}$ is locally uniformly bounded above on $\C^d$. Now since $q_1(z)\equiv 1$ we have $B_n(z)\geq 1$ so 
$$\liminf_{n\to \infty} \frac{1}{2n}\log B_n(z)\geq 0 \ \hbox{for} \ z\in \C^d$$
which, together with (\ref{upbound}), proves (\ref{onK}). We use these estimates to prove our next result.

\begin{proposition} \label{bnasymp} Let $\{q_j\}_j$ be $Z-$asymptotically Chebyshev for $K$. Then for the sequence $\{B_n\}$ in (\ref{bn}) we have the following: given any subsequence $Y\subset \N$, there is a further subsequence $Y_0\subset Y$ and a countable dense set of points $\{w_r\}$ in $\C^d$ such that 
$$\lim_{n\to \infty, \ n\in Y_0} \frac{1}{2n}\log B_n(w_r)=V_K(w_r), \ r=1,2,...$$

\end{proposition}

\begin{proof} Fix a subsequence $Y\subset \N$. Equation (\ref{onK}) shows the pointwise convergence of the full sequence $\{\frac{1}{2n}\log B_n\}$ to $0$ on $\hat K$ so we only need to find a subsequence $Y_0\subset Y$ and an appropriate countable dense set of points $\{w_r\}$ in $\C^d\setminus \hat K$. For the subsequence $Y$, from Definition \ref{good} the collection of polynomials 
$\{ q_{n,\alpha}\}_{n\in Y, \ |\alpha|=n}$ is asymptotically Chebyshev for $K$ so by Theorem 4.2 of \cite{Bloom} for $z\in \C^d \setminus \hat K$ we have
$$[\limsup_{n \to \infty, \ n\in Y, \ |\alpha|=n} \frac{1}{n} \log \frac{|q_{n,\alpha}(z)|}{||q_{n,\alpha}||_K}]^* = V_K(z).$$ 
Hence
\begin{equation} \label{five} \limsup_{n \to \infty, \ n\in Y, \ |\alpha|=n} \frac{1}{n} \log \frac{|q_{n,\alpha}(z)|}{||q_{n,\alpha}||_K} =\limsup_{n \to \infty, \ n\in Y, \ |\alpha|=n} \frac{1}{n} \log |p_{n,\alpha}(z)|= V_K(z)\end{equation}
on $\C^d \setminus \hat K$ except perhaps a pluripolar set.

Let $\{D_r\}$ be a countable collection of open balls in $\C^d \setminus \hat K$ such that choosing one point in each $D_r$ yields a countable dense set of points in $\C^d \setminus \hat K$. Note that each $D_r$ is nonpluripolar. Starting with $D_1$, choose $w_1\in D_1$ so that (\ref{five}) holds at $w_1$. Then for some subsequence of pairs $(n_1,\alpha_1), (n_2,\alpha_2),...$ with $n_j\in Y$ we have
$$\lim_{s \to \infty} \frac{1}{n_s} \log |p_{n_s,\alpha_s}(w_1)|= V_K(w_1)$$
where we can assume $n_j \leq n_{j+1}$. Let $Y_1=\{n_1,n_2,...\}\subset Y$. We claim that 
\begin{equation} \label{six}\lim_{n\to \infty, \ n\in Y_1} \frac{1}{2n}\log B_n(w_1)=V_K(w_1).\end{equation}
Here (\ref{six}) follows from (\ref{upbound}), the previous equality, and the elementary fact that for any $n\in \N$ and any $z\in \C^d$,
$$\frac{1}{2n}\log B_{n}(z)\geq \frac{1}{n} \log |p_{n,\alpha}(z)|.$$

We now repeat this argument with a point $w_2\in D_2$ to get a subsequence $Y_2 \subset Y_1$ with
$$\lim_{n\to \infty, \ n\in Y_2} \frac{1}{2n}\log B_n(w_2)=V_K(w_2).$$
Continuing, we get a countable family of nested subsequences $Y\supset Y_1 \supset Y_2 ...$ and points $w_r \in D_r$ with
$$\lim_{n\to \infty, \ n\in Y_r} \frac{1}{2n}\log B_n(w_r)=V_K(w_r).$$
Then using a diagonalization argument, we get a sequence $Y_0$ which is a subsequence of each $Y_r$ (except for the first $r$ terms) so that $Y_0 \subset Y$ and for each $r=1,2,...$ we have
$$\lim_{n\to \infty, \ n\in Y_0} \frac{1}{2n}\log B_n(w_r)=V_K(w_r).$$

\end{proof}

\begin{remark} \label{2.3} Suppose $d=1$. The condition that $\{q_j\}_{j=0,1,...}$ (where $j=deg(q_j)$) be asymptotically Chebyshev for $K$ is related to the $L^{\infty}-$version in \cite{DD} of asymptotically minimal. Moreover, the conclusion in Proposition \ref{bnasymp} cannot be strengthened to attain pointwise convergence of the full sequence $\{\frac{1}{2n}\log B_n\}$ to $V_K$ on $\C \setminus \hat K$. Indeed, fix $a\in \C \setminus \hat K$ and let $q_j(z)=(z-a)t_{j-1}(z)$ where $t_{j-1}(z)=z^{j-1}+...$ is the Chebyshev polynomial of degree $j-1$ for $K$. Setting $p_0(z)=q_0(z)\equiv 1$ we have $B_n(a)=\sum_{j=0}^n |p_j(a)|^2=1$ so that 
$$\lim_{n\to \infty} \frac{1}{2n}\log B_n(a)=0 < V_K(a).$$
We thank Bela Nagy for this example. In fact, it is not clear if Proposition \ref{bnasymp} can be strengthened to obtain a countable dense set of points in $\C^d$ such that the full sequence $\{\frac{1}{2n}\log B_{n}\}_{n \in \N}$ converges to $V_K$ at each point. See section 4 for the relevance of this question.

\end{remark}

As a corollary of Proposition \ref{bnasymp}, we will get the conclusion that the full sequence $\{ \frac{1}{2n}\log B_n\}_{n\in \N}$ converges in $L^1_{loc}(\C^d)$ to $V_K$. We utilize the following result on subharmonic functions, stated for $\C=\R^2$ as Theorem 2.4 in \cite{BD}. A probabilistic version of this result, Proposition \ref{keypropB}, will be crucial in proving probabilistic results in sections 3 and 4.

\begin{proposition} \label{keyprop} Let $\{u_n\}$ be a sequence of subharmonic functions on $D\subset \R^m, \ m\geq 2$ which are locally uniformly bounded above on $D$ and let $u$ be subharmonic and continuous in $D$. Suppose that 
$$[\limsup_{n\to \infty}u_n(w)]^*\leq u(w), \ w\in D$$
and there is a countable dense set of points $\{w_r\}$ in $D$ such that 
$$\lim_{n\to \infty}u_n(w_r)=u(w_r), \ r=1,2,...$$ 
Then $u_n\to u$ in $L^1_{loc}(D)$.

\end{proposition}

\begin{corollary} \label{bnloc} Let $\{q_j\}_j$ be $Z-$asymptotically Chebyshev for $K$. Then $\frac{1}{2n}\log B_n \to V_K$ in $L^1_{loc}(\C^d)$.

\end{corollary}

\begin{proof} We show that for any subsequence $Y_1\subset \N$, there exists a subsequence $Y_2 \subset Y_1$ such that 
\begin{equation}\label{imp1}\lim_{n\to \infty, \ n\in Y_2} \frac{1}{2n}\log B_n=V_K \ \hbox{in} \ L^1_{loc}(\C^d).\end{equation}
For this it suffices to observe that for any subsequence $Y\subset \N$, we have 
\begin{equation}\label{imp2} [\limsup_{n\to \infty, \ n\in Y} \frac{1}{2n}\log B_n(z)]^*\leq V_K(z), \ z\in \C^d \end{equation}
which follows immediately from (\ref{upbound}). For then, by Proposition \ref{bnasymp}, given a subsequence $Y_1\subset \N$, we get a subsequence $Y_2\subset Y_1$ and a countable dense set $\{w_r\}$ in $\C^d$ such that 
$$\lim_{n\to \infty, \ n\in Y_2} \frac{1}{2n}\log B_n(w_r)=V_K(w_r), \ r=1,2,...$$
Using (\ref{imp2}) for $Y=Y_2$, we can apply Proposition \ref{keyprop} with $D=\C^d$ and $u=V_K$ to conclude (\ref{imp1}) (recall that $\{\frac{1}{2n}\log B_n\}$ is locally uniformly bounded above on $\C^d$). 

\end{proof}

\begin{remark} Considering families $\{q_j\}$ with one $q_j=q_{n,\alpha}$ for each multiindex $\alpha \in \N^d$ with $q_{n,\alpha}(z)=c_{\alpha} z^{\alpha} + ...$ where $c_{\alpha}\not = 0$, another class of examples of asymptotically Chebyshev families for $K$ which includes a multidimensional generalization of Leja polynomials is given in \cite{Bloom} (see Corollary 4.5 and equation (4.28)). We do not know if such families are $Z-$asymptotically Chebyshev for $K$.

\end{remark}

\subsection{Local uniform convergence of $\{\frac{1}{2n}\log B_n$\}} In the  case where $q_j\in P(\alpha(j))$ is the $L^2(\mu)-$minimal (orthogonal) polynomial in $P(\alpha(j))$ for a regular compact set $K \subset \C^d$ where $\mu$ is (BM), we actually have $\lim_{n\to \infty} \frac{1}{2n}\log B_n(z)=V_K(z)$ locally uniformly in $\C^d$ where $p_j=q_j/||q_j||_K$. Indeed, in this subsection, we analyze more general conditions under which we have such a conclusion. We call a compact set $K\subset \C^d$ {\it locally regular} if for any $z_0 \in K$, $V_{K\cap \bar{B(z_0,r)}}$ is continuous at $z_0$ for $r>0$ where $B(z_0,r):=\{z:|z-z_0|< r\}$. We also let $Q$ be a continuous, real-valued function on $K$ (a continuous {\it weight}). The weighted extremal function for $K,Q$ is defined as
$$V_{K,Q}(z):=\sup\{u(z): u\in L(\C^d), \ u\leq Q \ \hbox{on} \ K\}.$$
It is known that 
\begin{enumerate}
\item for $K$ locally regular and $Q$ continuous, $V_{K,Q}$ is continuous on $\C^d$ (cf., \cite{NQD}); and
\item for $K$ compact,
$$V_{K,Q}(z)= \sup \{\frac{1}{deg(p)}\log |p(z)|: p\in \cup_n \mathcal P_n, \ ||pe^{-deg(p)\cdot Q}||_K\leq 1\}.$$ 

\end{enumerate}
Given $K$ and $Q$, for each $n=1,2,...$ define
$$ \phi_n(z):=\sup \{ |p(z)|: p\in  \mathcal P_n, \ \|pe^{-nQ}\|_K\leq 1\}.$$
If $K$ is locally regular and $Q$ is continuous, then
\begin{equation}\label{phin} \lim_{n\to \infty} \frac{1}{n}\log \phi_n(z) = V_{K,Q}(z)\end{equation} 
locally uniformly on $\C^d$ \cite{Bloom2}. The conclusion remains true if $Q=0$ (the unweighted case) for $K$ (globally) regular, i.e., when $V_K$ is continuous. For $n$ fixed we consider a basis $\{p_{nj}\}_{j=1,...,m_n}$ for $\mathcal P_n$ normalized so that $||p_{nj}e^{-nQ}||_K=1$. Then $|p_{nj}(z)|\leq \phi_n(z)$ so that
\begin{equation}\label{rhsin} B_n(z):=\sum_{j=1}^{m_n}|p_{nj}(z)|^2\leq m_n\cdot [\phi_n(z)]^2.\end{equation}
Now assume that $\{p_{nj}\}_{j=1,...,m_n}$ is an orthogonal basis in $L^2(e^{-2nQ}\mu_n)$ for $\mathcal P_n$ where $\mu_n$ is a probability measure on $K$. Let $M_n$ be the smallest constant such that 
$$||qe^{-nQ}||_K\leq M_n ||q||_{L^2(e^{-2nQ}\mu_n)}=M_n||qe^{-nQ}||_{L^2(\mu_n)}   \ \hbox{for all} \ q \in \mathcal P_n.$$
Note then $||p_{nj}||^2_{L^2(e^{-2nQ}\mu_n)}\geq 1/M_n, \ j=1,...,m_n$. Let $p$ be a polynomial of degree at most $n$ with $||pe^{-nQ}||_K\leq 1$. Writing $p(z)=\sum_{j=1}^{m_n}a_j \frac{p_{nj}(z)}{||p_{nj}||_{L^2(e^{-2nQ}\mu_n)}}$, 
$$|p(z)|^2\leq \sum_{j=1}^{m_n}|a_j|^2 \cdot \sum_{j=1}^{m_n}\frac{|p_{nj}(z)|^2}{||p_{nj}||^2_{L^2(e^{-2nQ}\mu_n)}}= ||p||_{L^2(e^{-2nQ}\mu_n)}^2\cdot \sum_{j=1}^{m_n}\frac{|p_{nj}(z)|^2}{||p_{nj}||^2_{L^2(e^{-2nQ}\mu_n)}}$$
$$ \leq ||pe^{-nQ}||_K^2\cdot M_n \sum_{j=1}^{m_n}|p_{nj}(z)|^2\leq M_n B_n(z).$$ 
Since $\phi_n(z)=\sup \{ |p(z)|: p\in  \mathcal P_n, \ \|pe^{-nQ}\|_K\leq 1\}$, taking the supremum over all such $p$ gives 
\begin{equation}\label{lhsin} [\phi_n(z)]^2\leq M_n B_n(z).\end{equation} 

From the above discussion, using (\ref{phin}), (\ref{rhsin}) and (\ref{lhsin}), we have the following result.

\begin{proposition}\label{locunif} Let $K\subset \C^d$ be a locally regular compact set, let $Q$ be a continuous, real-valued function on $K$, and let $\{\mu_n\}$ be a sequence of probability measures on $K$ such that, for $n=1,2,...$, we have
\begin{equation}\label{asymbm} ||qe^{-nQ}||_K\leq M_n||qe^{-nQ}||_{L^2(\mu_n)} \ \hbox{for all} \ q \in \mathcal P_n \ \hbox{with} \ \lim_{n\to \infty} M_n^{1/n}=1.\end{equation} 
Let $B_n(z):=\sum_{j=1}^{m_n}|p_{nj}(z)|^2$ where $\{p_{n1},...,p_{nm_n}\}$ is an orthogonal basis in $L^2(e^{-2nQ}\mu_n)$ for $\mathcal P_n$ with $||p_{nj}e^{-nQ}||_K=1$. Then 
$$\lim_{n\to \infty} \frac{1}{2n}\log B_n(z)=V_{K,Q} \ \hbox{locally uniformly on} \ \C^d.$$

\end{proposition}

We will call a sequence of probability measures $\{\mu_n\}$ satisfying (\ref{asymbm}) {\it asymptotically weighted Bernstein-Markov} (BM) for $K$ and $Q$. For the unweighted setting ($Q=0$) we have the corresponding result.

\begin{proposition}\label{locunifb} Let $K\subset \C^d$ be a regular compact set and let $\{\mu_n\}$ be a sequence of probability measures on $K$ such that, for $n=1,2,...$, we have
\begin{equation}\label{asymbmb} ||q||_K\leq M_n ||q||_{L^2(\mu_n)} \ \hbox{for all} \ q \in \mathcal P_n \ \hbox{with} \ \lim_{n\to \infty} M_n^{1/n}=1.\end{equation} 
Let $B_n(z):=\sum_{j=1}^{m_n}|p_{nj}(z)|^2$ where $\{p_{n1},...,p_{nm_n}\}$ is an orthogonal basis in $L^2(\mu_n)$ for $\mathcal P_n$ with $||p_{nj}||_K=1$. Then 
$$\lim_{n\to \infty} \frac{1}{2n}\log B_n(z)=V_K \ \hbox{locally uniformly on} \ \C^d.$$

\end{proposition}

We call $\{\mu_n\}$ satisfying (\ref{asymbmb}) {\it asymptotically Bernstein-Markov} (BM) for $K$. Here are some examples of such measures and corresponding basis polynomials $\{p_{n1},...,p_{nm_n}\}$.

\begin{enumerate} 

\item Let $K$ be regular and for each $n$, let $\zeta_{n1},...,\zeta_{nm_n}\in K$ be a set of {\it Fekete points of order $n$ for $K$}, i.e., 
\begin{equation}\label{vdm} |VDM(\zeta_{n1},...,\zeta_{nm_n})|=\max_{\zeta_1,...,\zeta_{m_n}\in K}|VDM(\zeta_1,...,\zeta_{m_n})|. \end{equation}
Define $\mu_n:=\frac{1}{m_n}\sum_{j=1}^{m_n}\delta_{\zeta_{nj}}$. Then the collection of fundamental Lagrange interpolating polynomials $l_{m_n1},...,l_{m_nm_n}\in \mathcal P_n$ for $\zeta_{n1},...,\zeta_{nm_n}$, i.e., $l_{m_nj}(\zeta_{nk})=\delta_{jk}$, satisfy $||l_{m_nj}||_K=1$ from (\ref{vdm}), and, by construction, these polynomials are an orthogonal basis in $L^2(\mu_n)$ for $\mathcal P_n$. Note that $||l_{m_nj}||_{L^2(\mu_n)}=\frac{1}{\sqrt m_n}$. For any $q\in \mathcal P_n$, 
$$q(z)=\sum_{j=1}^{m_n} q(\zeta_{nj})l_{m_nj}(z)=\sum_{j=1}^{m_n} \frac{1}{\sqrt m_n}q(\zeta_{nj})\sqrt{m_n} l_{m_nj}(z)$$ so that 
$$||q||_K \leq ||q||_{L^2(\mu_n)}\cdot m_n \sqrt{m_n}.$$
Thus, from Proposition \ref{locunifb}, setting $B_n(z):=\sum_{j=1}^{m_n}|l_{m_nj}(z)|^2$, we have $\frac{1}{2n}\log B_n \to V_K$ locally uniformly on $\C^d$.

\item Generalizing the previous example, given $K$ regular, for each $n$, let $\zeta_{n1},...,\zeta_{nm_n}\in K$ be a set of points in $K$ such that the corresponding sequence of Lebesgue constants 
$$\Lambda_{\alpha(m_n)}=\sup_{z\in K} \sum_{j=1}^{m_n}|l_{m_nj}(z)| \ \hbox{satisfy} \ \lim_{n\to \infty} \Lambda_{\alpha(m_n)}^{1/n}=1.$$
For Fekete points, $\Lambda_{\alpha(m_n)} \leq m_n$. A modification of the above argument for Fekete points shows that in this setting, $\mu_n:=\frac{1}{m_n}\sum_{j=1}^{m_n}\delta_{\zeta_{nj}}$ are asymptotically (BM) for $K$ and hence that $\frac{1}{2n}\log B_n \to V_K$ locally uniformly on $\C^d$ where $B_n(z):=\sum_{j=1}^{m_n}|l_{m_nj}(z)|^2$.

\item Let $\{\mathcal A_n\}_{n=1,...}$ be a {\it weakly admissible mesh} for $K$ where $\mathcal A_n= \{a_{n1},...,a_{ns_n}\}\in K$. This means that $s_n^{1/n}\to 1$ and $||p_n||_K\leq c_n ||p_n||_{\mathcal A_n}$ where $c_n^{1/n}\to 1$. Then the canonical measures $\mu_n:=\frac{1}{s_n}\sum_{j=1}^{s_n}\delta_{ns_j}$ are asymptotically (BM) for $K$. Here, since generally $s_n > > m_n$, one constructs the orthogonal polynomials $\{p_{n1},...,p_{nm_n}\}$ in $L^2(\mu_n)$ for $\mathcal P_n$ with $||p_{nj}||_K=1$ in a standard fashion, e.g., Gram-Schmidt. 
\end{enumerate}

\noindent This last example is due to F. Piazzon who introduced the notion of asymptotically Bernstein-Markov measures for $K$ \cite{F}. For more on admissible meshes, c.f., \cite{CL}.

We remark that Zaharjuta \cite{Z} used (\ref{zah}) to prove that the limit
$$\lim_{n\to \infty} |VDM(\zeta_{n1},...,\zeta_{nm_n})|^{1/l_n}=:d(K) \ \hbox{(transfinite diameter of $K$)}$$
exists for any nonpluripolar compact set $K$ where $\zeta_{n1},...,\zeta_{nm_n}\in K$ are Fekete points of order $n$ for $K$ and $l_n=\sum_{j=1}^{m_n}deg(e_j)$. In fact, he showed that $$\log d(K)=\frac{1}{|\Sigma|} \int_{\Sigma^o} \log \tau(K,\theta) d\theta$$
where $|\Sigma|$ denotes the $d-1$ (real) dimensional measure of $\Sigma$.


\section{Convergence in probability} We first consider random polynomials in $\C^d$ of degree at most $n$ of the form (\ref{probhb}); i.e.,
$$H_n(z):=\sum_{j=1}^{m_n} a_{nj}p_{nj}(z)$$
where $\{a_{nj}\}_{j=1,...,m_n; \ n=1,2,...}$ are an {\it array} of i.i.d. random variables and $\{p_{nj}\}_{j=1,...,m_n}$ are a basis for $\mathcal P_n$. We work in the corresponding probability space $(\Omega, \mathcal H,{\mathbb P})$. Sometimes to emphasize the random nature of our objects, we will include the variable $\omega\in \Omega$; e.g., we may write $a_j(\omega), \ H_n(z,\omega)$, etc. We assume $p_{n1}\equiv 1$ and we define, for $n=1,2,...$, the function 
\begin{equation} \label{bn2} B_n(z):=\sum_{j=1}^{m_n} |p_{nj}(z)|^2. \end{equation}

We recall the definition of the {\it concentration function} $\mathcal Q$ of a complex-valued random variable $X$: $\mathcal Q(X,r):= \sup_{z\in \C}{\mathbb P}(X \in B(z,r))$. This satisfies the following elementary properties: 
\begin{enumerate}
\item for any such $X$ and $a\in \C \setminus \{0\}$, $\mathcal Q(aX, r)=\mathcal Q(X, r/|a|)$; 
\item if $X$ are $Y$ are independent, $\mathcal Q(X+Y,r)\leq \mathcal Q(X,r)$. 
\end{enumerate}
If, e.g., $X$ has a bounded density (with respect to Lebesgue measure), then $\mathcal Q(X)\leq Cr^2$ for some $C>0$. 
We state our convergence in probability results in this array setting in $\C^d, \ d\geq 1$.

\begin{theorem} \label{convprob} Let $\{a_{nj}\}$ be an array of i.i.d. random variables. We assume that
\begin{itemize}
\item[(i)] $\mathbb P(|a_{nj}|>e^{|z|})=o(1/|z|^d)$ and  
\item[(ii)] $\mathcal Q(a_{nj};r)\to 0 \ \hbox{as} \ r\to 0$.
\end{itemize}  
Let $K\subset \C^d$ be compact and regular. Suppose that that $||p_{nj}||_K\leq 1, \ j=1,...,m_n; n=1,2,...$ and with $\{B_n\}$ in (\ref{bn2}), given any subsequence $Y\subset \N$, there is a further subsequence $Y_0\subset Y$ and a countable dense set of points $\{w_r\}$ in $\C^d$ such that 
$$\lim_{n\to \infty, \ n\in Y_0} \frac{1}{2n}\log B_n(w_r)=V_K(w_r), \ r=1,2,...$$
Then, in probability,
$$\frac{1}{n}\log |H_n|\to V_K \ \hbox{in} \ L^1_{loc}(\C^d).$$
\end{theorem}

\noindent Here bases $\{p_{nj}\}$ constructed from asymptotically (BM) measures as in Proposition \ref{locunifb} provide an example where we have the much stronger property that $\frac{1}{2n}\log B_n\to V_K$ locally uniformly on $\C^d$. 

\begin{theorem} \label{convprobQ} Let $\{a_{nj}\}$ be an array of i.i.d. random variables. We assume that
\begin{itemize}
\item[(i)] $\mathbb P(|a_{nj}|>e^{|z|})=o(1/|z|^d)$ and  
\item[(ii)] $\mathcal Q(a_{nj};r)\to 0 \ \hbox{as} \ r\to 0$.
\end{itemize}  
Let $K\subset \C^d$ be compact and locally regular and let $Q$ be a continuous, real-valued function on $K$. Suppose that that $||p_{nj}e^{-nQ}||_K\leq 1, \ j=1,...,m_n; n=1,2,...$ and with $\{B_n\}$ in (\ref{bn2}), given any subsequence $Y\subset \N$, there is a further subsequence $Y_0\subset Y$ and a countable dense set of points $\{w_r\}$ in $\C^d$ such that 
$$\lim_{n\to \infty, \ n\in Y_0} \frac{1}{2n}\log B_n(w_r)=V_{K,Q}(w_r), \ r=1,2,...$$
Then, in probability,
$$\frac{1}{n}\log |H_n|\to V_{K,Q} \ \hbox{in} \ L^1_{loc}(\C^d).$$
\end{theorem}

\noindent Here bases $\{p_{nj}\}$ constructed from asymptotically weighted (BM) measures as in Proposition \ref{locunif} provide an example. For $d=1$ and the case where $\{p_{nj}\}$ are an orthonormal basis in $L^2(\mu)$ for $\mathcal P_n$ for a single weighted (BM) measure $\mu$ (thus $\frac{1}{2n}\log B_n\to V_K$ locally uniformly on $\C$), Dauvergne \cite{Du} obtains the conclusion without the hypothesis on the concentration function.

\begin{remark} \label{issue} As a corollary, for the normalized zero currents $$dd^c (\frac{1}{n}\log |H_n|):=\frac{i}{\pi} \partial \bar \partial (\frac{1}{n}\log |H_n|) $$ (here, $dd^c \log |H_n|$ is the current of integration on the algebraic hypersurface $\{H_n=0\}$), we conclude that every subsequence $Y\subset \N$ has a further subsequence $Y_0\subset Y$ such that almost surely
$$\lim_{n\to \infty, \ n\in Y_0} dd^c (\frac{1}{n}\log |H_n|)\to dd^c V_K \ (\hbox{or} \ dd^cV_{K,Q}) \ \hbox{as currents}.$$
\end{remark}

The proofs of Theorems \ref{convprob} and \ref{convprobQ} are identical; we use the notation and give the proof of the former. We utilize a multivariate version of Lemma 3.1 of \cite{BD}, this time stated for random polynomials as in (\ref{probhb}). We observe that $m_n=0(n^d)$; hence we have the condition $\mathbb P(|a_j|>e^{|z|})=o(1/|z|^d)$ (instead of $o(1/|z|)$ when $d=1$ in \cite{BD}).

\begin{lemma} \label{3.1} Let $\{a_{nj}\}$ be i.i.d. complex random variables satisfying $\mathbb P(|a_{nj}|>e^{|z|})=o(1/|z|^d)$, and let $H_n(z)=\sum_{j=1}^{m_n} a_{nj}p_{nj}(z)$ be random polynomials where $||p_{nj}||_K\leq 1$. For any sequence $Y\subset \N$ there is a further subsequence $Y' \subset Y$ such that almost surely, the family $\{\frac{1}{n}\log |H_n|\}_{n\in Y'}$ is locally uniformly bounded above and for all $z\in \C^d$ 
\begin{equation} \label{3.1c} \limsup_{n\to \infty, \ n\in Y'} \frac{1}{n}\log |H_n(z)| \leq V_K(z).\end{equation}
\end{lemma}

\begin{proof} There are only two modifications to the proof of Lemma 3.1 in \cite{BD}. First, the condition $\mathbb P(|a_{nj}|>e^{|z|})=o(1/|z|^d)$ yields that for $\epsilon >0$ we have $\mathbb P (|a_{nj}| > e^{\epsilon |z|} )=o(1/|z|^d)$ and hence letting
$$\Omega_{n,\epsilon}:=\{\omega \in \Omega: |a_{nj}(\omega)| \leq e^{\epsilon n}, \ j=1,...,m_n\}$$
we have
$$\mathbb P(\Omega^c_{n,\epsilon})\leq \sum_{j=1}^{m_n} \mathbb P (|a_{nj}(\omega)| > e^{\epsilon n})\leq m_n o(1/n^d)\to 0$$
as $n\to \infty$ since $m_n=0(n^d)$. Thus for any sequence $Y\subset \N$ there is a further subsequence $Y_0=\{n_s\}_s \subset Y$ with 
$\sum_{s=1}^{\infty} \mathbb P(\Omega^c_{n_s,\epsilon})<\infty$. By Borel-Cantelli, almost surely there exists $s_0(\epsilon)$ so that for $s\geq s_0(\epsilon)$
$$|a_{n_sj}(\omega)|\leq e^{\epsilon n_s}, \ j=1,2,...,m_{n_s}.$$
Next, for $j=1,...,m_n$, since $||p_{nj}||_K\leq 1$ we have  
$$|p_{nj}(z)|\leq e^{deg(p_{nj})V_K(z)} \leq e^{nV_K(z)} \ \hbox{so that} \ \frac{1}{n}\log|p_{nj}(z)|\leq V_K(z), \ z\in \C^d;$$
and hence using this (in place of local uniform convergence of $\{\frac{1}{2n}\log |B_n|\}$, which we don't necessarily have) yields almost surely 
$$\frac{1}{n_s}\log |H_{n_s}(z)|\leq \frac{1}{n_s}\log \bigl(\sum_{j=1}^{m_{n_s}} |a_{n_sj} p_{n_sj}(z)|\bigr)\leq \epsilon +V_K(z), \ z \in \C^d$$
for $s \geq s_0(\epsilon)$. Taking a sequence $\{\epsilon_k\}$ with $\epsilon_k \searrow 0$, repeating this argument, for each $k=1,2,...$ we can find a subsequence $Y_k=\{n_{s,k}\}_s$ with $Y_k \subset Y_{k-1} $ so that almost surely  
$$\frac{1}{n_{s,k}}\log |H_{n_{s,k}}(z)| \leq \frac{1}{n_{s,k}}\log \bigl(\sum_{j=1}^{m_{n_{s,k}}} |a_{n_{s,k}j} p_{n_{s.k}j}(z)|\bigr)\leq  \epsilon_k +V_K(z), \ z \in \C^d$$
for $s \geq s_k(\epsilon_k)$. Taking a diagonal-like subsequence $Y'=\{n(k)\}_k \subset Y$ where 
$$n(k)=n_{s,k} \ \hbox{with} \ s \geq s_k(\epsilon_k) \ \hbox{and} \  \lim_{k\to \infty} n(k)= \infty$$ 
we have 
$$ \frac{1}{n(k)}\log |H_{n(k)}(z)| \leq \frac{1}{n(k)}\log \bigl(\sum_{j=1}^{m_{n(k)}} |a_{n(k)j} p_{n(k)j}(z)|\bigr)\leq  \epsilon_k +V_K(z), \ z \in \C^d$$
so that almost surely
$$\limsup_{k \to \infty} \frac{1}{n(k)}\log |H_{n(k)}(z)| \leq V_K(z), \ z \in \C^d.$$
This yields the local uniform boundedness of the family $\{\frac{1}{n(k)}\log |H_{n(k)}|\}$ as well as (\ref{3.1c}).

\end{proof}

We state without proof a multidimensional version of Theorem 4.1 \cite{BD}; this is a probabilistic version of Proposition \ref{keyprop} in our setting.

\begin{proposition} \label{keypropB} Let $\{H_n(z,\omega)\}$ be a sequence of random polynomials of the form (\ref{probhb}) such that, almost surely, 
\begin{enumerate}
\item $\{H_n(z,\omega)\}$ are locally uniformly bounded above on $\C^d$;  
\item $[\limsup_{n\to \infty}\frac{1}{n}\log |H_n(z,\omega)|]^*\leq V_K(z), \ z \in \C^d$; and 
\item there is a countable dense set of points $\{w_r\}$ in $\C^d$ such that 
$$\lim_{n\to \infty}\frac{1}{n}\log |H_n(w_r,\omega)|=V_K(w_r), \ r=1,2,...$$ 
\end{enumerate}
Then almost surely $\frac{1}{n}\log |H_n(z,\omega)|\to V_K$ in $L^1_{loc}(\C^d)$.

\end{proposition}

Here, as in Theorem \ref{convprob}, we are assuming $K$ is regular so $V_K$ is continuous. We use this to prove Theorem \ref{convprob}.

\begin{proof} To show that
$$\frac{1}{n}\log |H_n|\to V_K \ \hbox{in} \ L^1_{loc}(\C^d)$$ 
in probability, we must show that for any sequence $Y\subset \N$ there is a further subsequence $Y_1\subset Y$ such that  
$$\lim_{n\to \infty, \ n\in Y_1} \frac{1}{n}\log |H_n|= V_K \ \hbox{in} \ L^1_{loc}(\C^d)$$
almost surely. From Lemma \ref{3.1}, we know that for any sequence $Y\subset \N$ there is a further subsequence $Y_0\subset Y$ such that almost surely, the family $\{\frac{1}{n}\log |H_n|\}_{n\in Y_0}$ is locally uniformly bounded above and for all $z\in \C^d$ 
$$\limsup_{n\to \infty, \ n\in Y_0} \frac{1}{n}\log |H_n(z)| \leq V_K(z).$$

We use our hypothesis on $\{B_n\}$ to find a subsequence $Y_1\subset Y_0$ and a countable dense set of points $\{w_r\}$ in $\C^d$ such that 
\begin{equation} \label{bncon} \lim_{n\to \infty, \ n\in Y_1} \frac{1}{2n}\log B_n(w_r)=V_K(w_r), \ r=1,2,...\end{equation}
Following the reasoning in Theorem 4.2 of \cite{BD}, using Proposition (\ref{keypropB}) and a Cantor diagonalization procedure, it suffices, to finish the proof, to show that for each $w_r$, 
$$\lim_{n\to \infty} \frac{1}{n}\log |H_n(w_r)| = V_K(w_r) \ \hbox{in probability}.$$
To achieve this, it suffices to show that for each $w_r$, which we simply write as $w$, we have, for every $\epsilon >0$,
\begin{equation} \label{inprob1}
 \lim_{n\in Y_1, \ n\to \infty} \mathbb P\bigl( \frac{1}{n}\log |H_n(w)| > V_K(w)+\epsilon \bigr) =0
\end{equation}
and
\begin{equation} \label{inprob2}
 \lim_{n\in Y_1, \ n\to \infty} \mathbb P\bigl( \frac{1}{n}\log |H_n(w)| < V_K(w)-\epsilon \bigr) =0.
\end{equation}

The proof of (\ref{inprob1}) only uses the hypothesis $\mathbb P(|a_{nj}|>e^{|z|})=o(1/|z|^d)$. This condition implies that for $\epsilon >0$ we have $\mathbb P (|a_{nj}| > e^{\epsilon |z|} )=o(1/|z|^d)$ and hence letting
$$\Omega_{n,\epsilon}:=\{\omega \in \Omega: |a_{nj}(\omega)| \leq e^{\epsilon n/3}, \ j=1,...,m_n\}$$
we have
$$\mathbb P(\Omega^c_{n,\epsilon})\leq \sum_{j=1}^{m_n} \mathbb P (|a_{nj}(\omega)| > e^{\epsilon n/3})\leq m_n o(1/n^d)\to 0$$
as $n\to \infty$ since $m_n=0(n^d)$. Since
$$\frac1n\log|H_n(w)|=\frac{1}{n}\log |\sum_{j=1}^{m_n} a_{nj} p_{nj}(z)| \leq \frac{1}{2n}\log \sum_{j=1}^{m_n} |a_{nj}|^2+\frac{1}{2n}\log B_n(w),$$
for $\omega \in \Omega_{n,\epsilon}$ we have 
$$\frac{1}{2n}\log \sum_{j=1}^{m_n} |a_{nj}|^2 \leq \frac{1}{2n}\log (m_ne^{2\epsilon n/3})=  \frac{1}{2n}\log m_n + \epsilon/3 < \epsilon/2$$
for $n$ sufficiently large. By (\ref{bncon}) for $n\in Y_1$ sufficiently large we have
$$\frac{1}{2n}\log B_n(w) <V_K(w) +\epsilon/2;$$
these last two estimates yielding (\ref{inprob1}).

We next prove (\ref{inprob2}). Fom the hypothesis that $\lim_{n \in Y_1, \ n\to \infty} \frac{1}{2n}\log B_n(w)=V_K(w)$, given $\eta >0$, there exists $n_0=n_0(w,\eta)$ such that for all $n\geq n_0$ for some  
$j_n\in \{1,...,m_n\}$ we have 
\begin{equation}\label{last?} |p_{nj_n}(w)|\geq \frac{1}{\sqrt m_n}e^{n(V_K(w)-\eta)}.\end{equation}
Given $\epsilon >0$, using properties of $\mathcal Q$, we have
$${\mathbb P} (\frac1n\log|H_n(w)|\leq V_K(w)-\epsilon)= {\mathbb P} ( |\sum_{j=1}^{m_n} a_{nj} p_{nj}(w)| \leq e^{n(V_K(w)-\epsilon)})$$
$$\leq \mathcal Q(\sum_{j=1}^{m_n} a_{nj} p_{nj}(w), e^{n(V_K(w)-\epsilon)})\leq \mathcal Q (a_{nj_n} p_{nj_n}(w), e^{n(V_K(w)-\epsilon)})=\mathcal Q(a_{nj_n}, \frac{ e^{n(V_K(w)-\epsilon)}}{|p_{nj_n}(w)|}).$$
Now using (\ref{last?}), for $n\in Y_1$ with $n\geq n_0(w,\eta)$ we have
$$\frac{ e^{n(V_K(w)-\epsilon)}}{|p_{nj_n}(w)|}\leq \sqrt m_n \cdot e^{n(\eta -\epsilon)}.$$
Thus, given $\epsilon >0$, by taking $\eta <\epsilon$, since $m_n=0(n^d)$, we have 
$$\frac{ e^{n(V_K(w)-\epsilon)}}{|p_{nj_n}(w)|}\to 0 \ \hbox{as} \ n\to \infty, \ n\in Y_1.$$
Using our hypothesis that $\mathcal Q(a_{nj};r)\to 0$ as $r\to 0$ yields (\ref{inprob2}). 
\end{proof}

We can eliminate the assumption on the concentration function in Theorem \ref{convprob} if we consider random polynomials of degree at most $n$ of the form (\ref{probh}); i.e., $ H_n(z):=\sum_{j=1}^{m_n} a_jp_j(z)$, where $\{a_j\}$ is a {\it sequence} of i.i.d. complex random variables and $\{p_j\}_{j=1,...,m_n}$ form a basis for $\mathcal P_n$ with $||p_j||_K\leq 1$. Here $K$ is a regular compact set in $\C^d$. This generalizes Theorem 7.6 in \cite{BD}.

\begin{theorem} \label{convprobs} Let $\{B_n\}$ in (\ref{bn}) satisfy the conclusion in Proposition \ref{bnasymp}: given any subsequence $Y\subset \N$, there is a further subsequence $Y_0\subset Y$ and a countable dense set of points $\{w_r\}$ in $\C^d$ such that 
$$\lim_{n\to \infty, \ n\in Y_0} \frac{1}{2n}\log B_n(w_r)=V_K(w_r), \ r=1,2,...$$
Let $\{a_j\}$ satisfy $\mathbb P(|a_j|>e^{|z|})=o(1/|z|^d)$
Then, in probability,
$$\frac{1}{n}\log |H_n|\to V_K \ \hbox{in} \ L^1_{loc}(\C^d).$$
\end{theorem}


\begin{remark} For $\{q_j\}_j$ $Z-$asymptotically Chebyshev for $K$, the corresponding sequence $\{p_j:=q_j/||q_j||_K\}_j$ gives a basis such that $\{B_n\}$ satisfy the conclusion in Proposition \ref{bnasymp}. 
\end{remark}

 The following is a multivariate version of Theorem 5.2 of \cite{BD}; it is a restatement of Theorem 7.5 of \cite{BD}. The key ingredient is a version of the Kolmogorov-Rogozin inequality stated as Theorem 5.1 of \cite{BD}. It is important to observe that this inequality applies to {\it sequences} of random coefficients but not to arrays.

\begin{theorem} \label{KR} Let $\{ b_j \}$ be a sequence of nonzero complex numbers satisfying 
$$\lim_{n\to \infty}\frac{1}{2n}\log \sum_{j=1}^{m_n} |b_j|^2 =: B\geq 0.$$
Let $\{a_j\}$ be i.i.d. complex random variables satisfying $\mathbb P(|a_j|>e^{|z|})=o(1/|z|^d)$. Then 
$$\lim_{n\to \infty}\frac{1}{n}\log \sum_{j=1}^{m_n} |a_j b_j|=B$$
in probability.

\end{theorem}


We now proceed with the proof of Theorem \ref{convprobs}.

\begin{proof} To show that
$$\frac{1}{n}\log |H_n|\to V_K \ \hbox{in} \ L^1_{loc}(\C^d)$$ 
in probability, we must show that for any sequence $Y\subset \N$ there is a further subsequence $Y_2\subset Y$ such that  
$$\lim_{n\to \infty, \ n\in Y_2} \frac{1}{n}\log |H_n|= V_K \ \hbox{in} \ L^1_{loc}(\C^d)$$
almost surely. From Lemma \ref{3.1}, for any sequence $Y\subset \N$ there is a further subsequence $Y_0\subset Y$ such that almost surely, the family $\{\frac{1}{n}\log |H_n|\}_{n\in Y_0}$ is locally uniformly bounded above and for all $z\in \C^d$ 
$$\limsup_{n\to \infty, \ n\in Y_0} \frac{1}{n}\log |H_n(z)| \leq V_K(z).$$
Next, using the hypothesis that the conclusion of Proposition \ref{bnasymp} holds, there is a further subsequence $Y_1\subset Y$ and a countable dense set of points $\{w_r\}$ in $\C^d$ such that 
$$\lim_{n\to \infty, \ n\in Y_1} \frac{1}{2n}\log B_n(w_r)=V_K(w_r), \ r=1,2,...$$
Thus the hypotheses of Theorem \ref{KR} hold with $b_j=p_j(w_r)$ and $B=V_K(w_r)$; note we may assume $\{w_r\}$ are chosen so that $p_j(w_r)\not = 0$ for all $j, r$. Hence for each $w_r$ we have 
$$\lim_{n\to \infty, \ n\in Y_1} \frac{1}{n}\log |H_n(w_r)| = V_K(w_r)$$
in probability. By passing to a further subsequence $Y_2\subset Y_1$, for each $w_r$ we have 
$$\lim_{n\to \infty, \ n\in Y_2} \frac{1}{n}\log |H_n(w_r)| = V_K(w_r)$$
almost surely. 

Then by Proposition \ref{keypropB} applied to the sequence $\{H_n\}_{n\in Y_2}$, almost surely 
$$\lim_{n\to \infty, \ n\in Y_2} \frac{1}{n}\log |H_n| = V_K \ \hbox{in}  \ L^1_{loc}(\C^d).$$

\end{proof}

\section{Almost sure convergence}

In this section, we focus on almost sure convergence. Given a basis $\{p_{nj}\}_{j=1,...,m_n}$ for $\mathcal P_n$, we form random polynomials 
$$H_n(z):=\sum_{j=1}^{m_n} a_{nj}p_{nj}(z)$$
as in (\ref{probhb}) where $\{a_{nj}\}$ form an array of i.i.d. random variables, and we define $B_n(z):=\sum_{j=1}^{m_n}|p_{nj}(z)|^2$. It may occur that $a_{nj}=a_j$ is independent of $n$ and $H_n$ are as in (\ref{probh}). The main results of this section are Theorem \ref{codim1} and its weighted analogue, Theorem \ref{codim1w}.

\begin{theorem}\label{codim1}
Let $\{a_{nj}\}$ be an array of i.i.d. random variables. We assume that
\begin{itemize}
\item[(i)] $\Bbb{E}[\big(\log(1+|a_{nj}|\big)^d]<\infty$.
\item[(ii)] For some $\gamma >0$, we have $\mathcal Q(a_{nj},r)\leq r^{\gamma}$ for $r\leq r_0$.
\end{itemize}  
Let $K\subset \C^d$ be compact and regular. Suppose that that $||p_{nj}||_K\leq 1, \ j=1,...,m_n; n=1,2,...$ and for a countable dense set of points $\{w_j\}$ in $\C^d$
$$\lim_{n\to \infty} \frac{1}{2n}\log B_n(w_j)=V_K(w_j) \ \hbox{for all} \ j=1,2,... $$
Then for random polynomials $H_n(z):=\sum_{j=1}^{m_n} a_{nj}p_{nj}(z)$, almost surely,
$$\frac{1}{n}\log |H_n|\to V_K \ \hbox{in} \ L^1_{loc}(\C^d)$$
and hence $\frac{1}{n}dd^c \log |H_n| \to dd^cV_{K}$ almost surely in the sense of currents.
\end{theorem}

\noindent Here bases $\{p_{nj}\}$ constructed from asymptotically (BM) measures as in Proposition \ref{locunifb} provide an example where we have the much stronger property that $\frac{1}{2n}\log B_n\to V_K$ locally uniformly on $\C^d$. 

We state the weighted analogue of Theorem \ref{codim1}.

\begin{theorem}\label{codim1w}
Let $\{a_{nj}\}$ be an array of i.i.d. random variables. We assume that
\begin{itemize}
\item[(i)] $\Bbb{E}[\big(\log(1+|a_{nj}|\big)^d]<\infty$.
\item[(ii)] For some $\gamma >0$, we have $\mathcal Q(a_{nj},r)\leq r^{\gamma}$ for $r\leq r_0$.
\end{itemize}  
Let $K\subset \C^d$ be compact and locally regular and let $Q$ be a continuous, real-valued function on $K$. Suppose that that $||p_{nj}e^{-nQ}||_K\leq 1, \ j=1,...,m_n; n=1,2,...$ and for a countable dense set of points $\{w_j\}$ in $\C^d$
$$\lim_{n\to \infty} \frac{1}{2n}\log B_n(w_j)=V_{K,Q}(w_j) \ \hbox{for all} \ j=1,2,... $$
Then for random polynomials $H_n(z):=\sum_{j=1}^{m_n} a_{nj}p_{nj}(z)$, almost surely,
$$\frac{1}{n}\log |H_n|\to V_{K,Q} \ \hbox{in} \ L^1_{loc}(\C^d)$$
and hence $\frac{1}{n}dd^c \log |H_n| \to dd^cV_{K,Q}$ almost surely in the sense of currents.
\end{theorem}

\noindent Here bases $\{p_{nj}\}$ constructed from asymptotically weighted (BM) measures as in Proposition \ref{locunif} provide an example.

We will use the following lemma (cf., the proof of \cite{BD}, Lemma 3.2).  

\begin{lemma}\label{lem1}
Let $\{a_j\}$ be a sequence of i.i.d. random variables on a probability space $(\Omega, \mathcal F,\Bbb{P})$ which satisfy $\Bbb{E}[\big(\log(1+|a_j|\big)^d]<\infty$. For each $\epsilon>0$ there exists a random constant $C=C(\omega)$ such that almost surely
$$|a_j|\leq C\exp(\sqrt[d]{j \epsilon }) \ \hbox{for all} \ j.$$ 
\end{lemma}
\begin{proof}
For a nonnegative random variable $X$ on $(\Omega, \mathcal F,\Bbb{P})$ we have
$$
\sum_{j=1}^{\infty}\Bbb{P}[X\geq j]\leq \Bbb{E}[X]\leq 1+\sum_{j=1}^{\infty}\Bbb{P}[X\geq j].
$$
Letting $X=\frac{1}{\epsilon}(\log(1+|a_1|))^d$ and using the assumption that $a_j$ are identically distributed, we obtain
$$\sum_{j=1}^{\infty}\Bbb{P}[ |a_j|\geq e^{\sqrt[d]{j\epsilon}}]<\infty.$$ 
 Hence, by the Borel-Cantelli lemma there exists a $j_0=j_0(\epsilon)$ so that 
$$|a_j|<e^{\sqrt[d]{j\epsilon}}  \ \hbox{for all} \ j\geq j_0.$$
Then we can find a a random constant $C$ such that almost surely
$$|a_j|<Ce^{\sqrt[d]{j\epsilon}}  \ \hbox{for all} \ j.$$

\end{proof}

We proceed with the proof of Theorem \ref{codim1}; the proof of Theorem \ref{codim1w} is entirely analogous.

 \begin{proof} 
 
 To show that almost surely $\frac1n\log|H_n| \to V_K $ in $L^1_{loc}(\Bbb{C}^d)$, we will utilize Proposition \ref{keypropB}. By Lemma \ref{lem1} for each $\epsilon>0$ there exists a random constant $C$ such that almost surely $|a_{nj}|\leq C e^{\sqrt[d]{\epsilon m_n}}, \ j=1,...,m_n$ so that
$$\sum_{j=1}^{m_n}|a_{nj}|^2\leq C^2m_ne^{2\sqrt[d]{\epsilon m_n}} \  \hbox{for all} \ n.$$
Now by Cauchy-Schwarz, recalling that $B_n(z)=\sum_{j=1}^{m_n}|p_{nj}(z)|^2$, we have
$$\frac1n\log|H_n(z)|=\frac{1}{n}\log |\sum_{j=1}^{m_n} a_{nj} p_{nj}(z)| \leq \frac{1}{2n}\log \sum_{j=1}^{m_n} |a_{nj}|^2+\frac{1}{2n}\log B_n(z).$$
As in the discussion in subsection 2.1, since $||p_{nj}||_K\leq 1$, by definition of $V_K$, $|p_j(z)|\leq e^{deg(p_j)V_K(z)}$ so that
$$B_n(z)\leq \sum_{j=1}^{m_n} e^{2deg(p_j)V_K(z)}\leq m_ne^{2nV_K(z)}$$
which shows that the sequence $\{\frac{1}{2n}\log B_n\}$ is locally uniformly bounded above. Together with the previous inequalities, this give the 
local uniform boundedness from above of $\{\frac1n\log|H_n|\}$ almost surely. Using $m_n=O(n^d)$ we deduce that 
 there exists a set $\mathcal{A}\subset \mathcal H$ of probability one such that for $\omega \in  \mathcal{A}$ and a sequence $\{H_n(z,\omega)\}$ we have 
$$\limsup_{n\to \infty}\frac1n\log|H_n(z)| \leq \limsup_{n\to \infty}\big(\frac{1}{2n}\log \sum_{j=1}^{m_n}|a_{nj}|^2 +\frac{1}{2n}\log B_n(z)\bigr)
\leq  \epsilon^{1/d}+V_{K}(z)$$
 on $\C^d$. We can repeat the above arguments for a sequence $\{\epsilon_k\}_k$ with $\epsilon_k \searrow 0$; for each $k$  there exists a set $\mathcal{A}_k\subset \mathcal H$ of probability one such that for $\omega \in  \mathcal{A}_k$ and a sequence $\{H_n(z,\omega)\}$ we have 
$$\limsup_{n\to \infty}\frac1n\log|H_n(z)| \leq \limsup_{n\to \infty}\big(\frac{1}{2n}\log \sum_{j=1}^{m_n}|a_{nj}|^2 +\frac{1}{2n}\log B_n(z)\bigr)
\leq  \epsilon_k^{1/d}+V_{K}(z).$$
 Since a countable union of sets of probability 
 zero has probability zero, we conclude that
$(\displaystyle\limsup_{n\to \infty}\frac1n\log|H_n(z)|)^*\in L(\Bbb{C}^d)$ and almost surely 
\begin{equation}\label{limsup}
\limsup_{n\to \infty}\frac1n\log|H_n(z)| \leq V_{K}(z)\ \text{for}\ z\in \Bbb{C}^d.
\end{equation} 

From the hypothesis that $\lim_{n\to \infty} \frac{1}{2n}\log B_n(w_j)=V_K(w_j)$ for each $w_j$, given $w_j$ and $\eta >0$, there exists $n_0=n_0(w_j,\eta)$ such that for all $n\geq n_0$ for some  
$j_n\in \{1,...,m_n\}$ we have 
\begin{equation}\label{last?} |p_{nj_n}(w_j)|\geq \frac{1}{\sqrt m_n}e^{n(V_K(w_j)-\eta)}.\end{equation}
Now given $\epsilon >0$, using properties 1. and 2. of $\mathcal Q$,
$${\mathbb P} (\frac1n\log|H_n(w_j)|\leq V_K(w_j)-\epsilon)= {\mathbb P} ( |\sum_{j=1}^{m_n} a_{nj} p_{nj}(w_j)| \leq e^{n(V_K(w_j)-\epsilon)})$$
$$\leq \mathcal Q(\sum_{j=1}^{m_n} a_{nj} p_{nj}(w_j), e^{n(V_K(w_j)-\epsilon)})\leq \mathcal Q (a_{nj_n} p_{nj_n}(w_j), e^{n(V_K(w_j)-\epsilon)})=\mathcal Q(a_{nj_n}, \frac{ e^{n(V_K(w_j)-\epsilon)}}{|p_{nj_n}(w_j)|}).$$
From the hypothesis, $\mathcal Q(a_{nj_n},r)\leq r^{\gamma}$ and (\ref{last?}) we have
$${\mathbb P} (\frac1n\log|H_n(w_j)|\leq V_K(w_j)-\epsilon)\leq \frac{ e^{\gamma n(V_K(w_j)-\epsilon)}}{|p_{nj_n}(w_j)|^{\gamma}}\leq m_n^{\gamma/2} \cdot e^{(\gamma n(\eta-\epsilon))}.$$
Thus, given $\epsilon >0$, by taking $\eta <\epsilon$, since $m_n=0(n^d)$, we have 
$$\sum_{n=1}^{\infty} {\mathbb P} (\frac1n\log|H_n(w_j)|\leq V_K(w_j)-\epsilon) <\infty.$$
By Borel-Cantelli there exists $N=N(w_j,\epsilon)$ such that, with probability one, $\frac1n\log|H_n(w_j)|> V_K(w_j)-\epsilon$ for $n>N$. Repeating this argument for a sequence $\{\epsilon_k\}$ with $\epsilon_k \searrow 0$ we conclude that almost surely we have
$$\liminf_{n\to \infty}\frac1n\log|H_n(w_j)|\geq V_{K}(w_j)$$
for this fixed $w_j$. Since a countable union of sets of probability zero is still of probability zero, applying this argument to our countable dense set of points $\{w_j\}$ in $\C^d$ we have 
$$\liminf_{n\to \infty}\frac1n\log|H_n(w_j)|\geq V_{K}(w_j), \ j=1,2,...$$
almost surely. Together with (\ref{limsup}) we conclude that almost surely
$$\lim_{n\to \infty}\frac1n\log|H_n(w_j)|=V_{K}(w_j), \ j=1,2,...$$
Proposition \ref{keypropB} gives that almost surely $\frac1n\log|H_n| \to V_K $ in $L^1_{loc}(\Bbb{C}^d)$.

\end{proof}

We make some remarks on Theorem \ref{codim1}.

\begin{enumerate}
\item Suppose $\{p_{j}=q_j/||q_j||_K\}$ where $\{q_j\}$ are $Z-$asymptotically Chebyshev for a regular compact set $K$. If we assume a stronger conclusion than  that in Proposition \ref{bnasymp}, namely, if 
$$\lim_{n\to \infty} \frac{1}{2n}\log B_n(w_r)=V_K(w_r), \ r=1,2,...$$
on a countable dense set of points for the {\it full} sequence $\{\frac{1}{2n}\log B_n\}$, then this basis satisfies the hypotheses in Theorem \ref{codim1}.

\item We conjecture that for $K$ regular and $H_n$ as in (\ref{probh}), the single hypothesis (i), $\Bbb{E}[\big(\log(1+|a_j|\big)^d]<\infty$, is sufficient for the conclusion that almost surely in $\mathcal{H}$ we have $\frac{1}{n}dd^c \log |H_n| \to dd^cV_{K} \ \text{as}\ n\to \infty$. For $d=2$ and the case of the bivariate Kac ensemble, this was proved in \cite{BD}, Theorem 7.7.

\item A strong converse to Theorem \ref{codim1} for $H_n$ as in (\ref{probh}) in $\C$ was given by Dauvergne \cite{DD}: for univariate random polynomials $H_n(z)=\sum_{j=0}^n a_j p_j(z)$ where $\{p_j\}$ are asymptotically minimal for $K\subset \C$ (see Remark \ref{2.3}) and $\{a_j\}$ are nondegenerate i.i.d. random variables, if $\Bbb{E}\big(\log(1+|a_j|\big)=\infty$, then the sequence $\mu_{H_n}$ (see the introduction) has no almost sure limit in the space of probability measures on $\C$. In $\C^d$ for $d\geq 3$ and $k=1$, some special cases of a converse to Theorem \ref{codim1} are known; see \cite{Bay1}.

\end{enumerate}



\vspace{1cm}
{\obeylines
\texttt{T. Bayraktar, tbayraktar@sabanciuniv.edu 
Sabanci University, Istanbul, Turkey
\medskip
T. Bloom, bloom@math.toronto.edu 
University of Toronto, Toronto, Ontario M5S 2E4 Canada
\medskip
N. Levenberg, nlevenbe@indiana.edu
Indiana University, Bloomington, IN 47405 USA }
}


\begin{thebibliography}{99}

\bibitem {ba} T. Bayraktar, Equidistribution of zeros of random holomorphic sections, \emph{Indiana U. Math. J.}, \textbf{65}, (2016), no. 5, 1759-1793.

\bibitem {Bay} T. Bayraktar, Zero distribution of random sparse polynomials, \emph{Mich. Math. J.}, \textbf{66}, (2017), no. 2, 389-419.

\bibitem{Bay1} T. Bayraktar,  On global universality for zeros of random polynomials, \emph{Hacet. J. Math. Stat.}, \textbf{48}, (2019), no. 2 384-398.





\bibitem{Bloom} T. Bloom, On families of polynomials which approximate the pluricomplex Green function, \emph{Indiana U. Math. J.}, \textbf{50}, (2001), no. 4, 1545-1566.

\bibitem{Bloom2}  T. Bloom, Random polynomials and (pluri-)potential theory, \emph{Ann. Polon. Math.}, \textbf{91} (2007), no. 2-3, 131-141.

 
 \bibitem{BMsur} T. Bloom, N. Levenberg, F. Piazzon and F. Wielonsky, Bernstein-Markov: a survey, \emph{Dolomites Research Notes on Approximation}, Vol. 8 (special issue) (2015), 75-91.

\bibitem{BD} T. Bloom and D. Dauvergne, Asymptotic zero distribution of random orthogonal polynomials, arXiv:1801.10125v2.

\bibitem{blrp} T. Bloom and N. Levenberg, Random polynomials and pluripotential-theoretic extremal functions, \emph{Potential Analysis}, \textbf{42} (2015) no. 2, 311-334.

\bibitem{BS} T. Bloom and B. Shiffman, Zeros of random polynomials on $\C^m$, \emph{Math. Res. Lett.}, \textbf{14} (2007), no. 3, 469-479. 

\bibitem{CL} J.-P. Calvi and N. Levenberg, Uniform approximation by discrete least squares polynomials, \emph{Journal of Approx. Th.}, \textbf{152}, (2008), no. 1, 82-100. 

\bibitem{DD} D. Dauvergne,  A necessary and sufficient condition for convergence of the zeros of random polynomials, \emph{Advances in Mathematics}, \textbf{384} (2021), 1-33.

\bibitem{Du} D. Dauvergne, private communication.




\bibitem{IZ} I. Ibragimov, and D. Zaporozhets, On distribution of zeros of random polynomials in complex plane, \emph{Prokhorov and contemporary probability theory}, Springer, 2013, pp. 303-323.

\bibitem{KZ} Z. Kabluchko and D. Zaporozhets, Asymptotic distribution of complex zeros of random analytic functions, \emph{Annals of Probability}, \textbf{42}, (2014), no. 4, 1374-1395.



\bibitem {NQD} Nguyen Quang Dieu, Regularity of certain sets in $\C^n$, \emph{Ann. Polon. Math.}, \textbf{82} (2003), no. 3, 219-232.

\bibitem{F} F. Piazzon, Bernstein Markov properties and applications, PhD thesis, University of Padova, 2016.

\bibitem{PR} I. Pritsker and K. Ramachandran, Equidistribution of zeros of random polynomials, \emph{J. Approx. Theory}, \textbf{215}, (2017), 106-117. 

\bibitem{SZ} B. Shiffman and S. Zelditch, Distribution of zeros of random and quantum chaotic sections of positive line bundles, \emph{Comm. Math. Phys.},  \textbf{200}, (1999), no. 3, 661-683. 


\bibitem{Z} V. P. Zaharjuta, Transfinite diameter, Chebyshev constants, and capacity for compacta in $\C^n$,  \emph{Math. USSR Sbornik}, \textbf{25} (1975), no. 3, 350-364. 


\end{thebibliography}
\end{document}